\documentclass[12pt,centertags,oneside]{amsart}
\usepackage{a4}
\usepackage{amsmath,amstext,amsthm,amssymb,amsfonts,amscd}
\usepackage{mathrsfs}
\usepackage[colorlinks=false]{hyperref}
\usepackage[capitalise]{cleveref}
\usepackage{latexsym,bm,graphicx}
\usepackage[T1]{fontenc}
\usepackage[latin1]{inputenc}
\usepackage{color,tocvsec2}
\usepackage{typearea}
\usepackage{charter}
\usepackage{enumerate}
\usepackage{lscape}
\usepackage[all]{xy}
\usepackage{tikz-cd}
\usepackage{rotating}
\usetikzlibrary{matrix,arrows,decorations.pathmorphing}

\usepackage{multicol}        
\makeindex             

\usepackage{geometry}
\theoremstyle{plain}
\newtheorem{theorem}{Theorem}[section]

\newtheorem{lemma}[theorem]{Lemma}
\newtheorem{corollary}[theorem]{Corollary}

\newtheorem{conjecture}[theorem]{Conjecture}

\theoremstyle{definition}
\theoremstyle{definition}

\theoremstyle{remark} 
\newtheorem{remark}[theorem]{Remark}

\numberwithin{equation}{section}
\numberwithin{figure}{section}

\makeatletter
\@namedef{subjclassname@2020}{\textup{2020} Mathematics Subject Classification}
\makeatother

\begin{document}

\title[Long neck principle and spectral width inequality]{A note on the long neck principle and spectral width inequality of geodesic collar neighborhoods} 
\author[D. Liu]{Daoqiang Liu}
\address{School of Mathematical Sciences\\
     Capital Normal University\\
     100048, Beijing\\
     China}
\email{\href{mailto:dqliumath@cnu.edu.cn}{dqliumath@cnu.edu.cn}}
\urladdr{\href{https://www.dqliu.cn}{www.dqliu.cn}}
\thanks{}

\date{}

\begin{abstract}
 The main purpose of this short note is to derive some generalizations of the long neck principle and give a spectral width inequality of geodesic collar neighborhoods. Our results are obtained via the spinorial Callias operator approach. An important step is to introduce the relative Gromov-Lawson pair on a compact manifold with boundary, relative to a background manifold. 
\end{abstract}

\maketitle

\section{Introduction}\label{intro}
The topology and geometry related to the scalar curvature have been a significant theme in Riemannian geometry. In addition to the minimal surface method developed by Schoen and Yau \cite{schoen1979onthe}, a tool to study metrics of positive scalar curvature is the index theory of Dirac operators in the spin setting. Lichnerowicz \cite{lichnerowicz1963spineurs} showed that on a closed spin manifold with positive scalar curvature, Dirac operator is invertible and hence its index must vanish. Using the Atiyah-Singer index theorem \cite{atiyah1963index}, $\widehat{A}$-genus must also vanish. Gromov and Lawson \cite{gromov1983positive} proved that a variety of so-called enlargeable manifolds do not admit any metric of positive scalar curvature by their relative index theorem for complete manifolds. Additionally, the sphere rigidity theorem of Llarull \cite{llarull1998sharp} illustrates the power of Dirac operators for compact Riemannian spin manifolds, which states that for a compact $m$-dimensional Riemannian spin manifold $(M,g_M)$ with ${\rm scal}_M\geq m(m-1)$ any area non-increasing map $\Phi: M\to S^m$ \footnote{We denote by $S^m$ the unit $m$-sphere equipped with its standard Riemannian metric.} of non-zero degree is a Riemannian isometry. Gromov proposed the following conjecture, which can be considered as an extension of Llarull's theorem to compact manifolds with boundary.
 
\begin{conjecture}[{\cite[p. 87, Long neck problem]{gromov2019four}}]\label{conj:long-neck-problem}
Let $(M,g_M)$ be a compact $m$\nobreakdash-dimensional Riemannian manifold with boundary with ${\rm scal}_M\geq m(m-1)$. Suppose that $\Phi:M\to S^m$ is a smooth area non-increasing map which is locally constant near the boundary. If the distance with respect to the metric $g_M$
\begin{equation}\label{eq:the-neck-length-inequality}
{\rm dist}_{g_M}({\rm supp}({\rm d}\Phi),\partial M)\geq \frac{\pi}{m},
\end{equation}
then ${\rm deg}(\Phi)=0$.
\end{conjecture}
Cecchini firstly solved the case of this conjecture when $M$ is spin, $\Phi$ is strictly area non-increasing and the neck length inequality \eqref{eq:the-neck-length-inequality} is strict via Dirac operators with compatible potentials. 
\begin{theorem}\cite[Theorem~A, Long neck principle]{cecchini2020along}\label{thm:cecchini2020longneck}
Let $(M,g_M)$ be a compact $m$-dimensional Riemannian spin manifold with boundary. Let $\Phi: M\to S^m$ be a smooth strictly area non-increasing map. If $m$ is odd, we make the further assumption that $\Phi$ is constant in a neighborhood of $\partial M$. Assume that ${\rm scal}_M\geq \sigma>0$. Moreover, ${\rm scal}_M\geq m(m-1)$ on ${\rm supp}({\rm d}\Phi)$ and 
\begin{equation}\label{eq:cecchini2020necklength}
{\rm dist}_{g_M}({\rm supp}({\rm d}\Phi),\partial M)>\pi\sqrt{\frac{m-1}{m\sigma}}.
\end{equation}
Then ${\rm deg}(\Phi)=0$.
\end{theorem}
Subsequently, Cecchini and Zeidler made further progress on this problem using Dirac operators with admissible potentials and the mean curvature of the boundary. 
\begin{theorem}[{\cite[Theorem~1.4]{cecchini2022scalar}}]\label{thm:CZ-result}
Let $(M,g_M)$ be a compact $m$-dimensional Riemannian spin manifold with non-empty boundary, $m\geq 2$ even, and let $\Phi:M\to S^m$ be a smooth area non-increasing map. Assume that ${\rm scal}_{M}\geq m(m-1)$. Moreover, suppose there exists $\ell\in (0,\pi/m)$ such that the mean curvature $H \geq -\tan(m\ell/2)$ and ${\rm dist}_{g_M}({\rm supp}({\rm d}\Phi),\partial M)\geq \ell$. Then 
\[
  {\rm deg}(\Phi)=0.
\]
\end{theorem}
The neck length inequality in Theorem \ref{thm:CZ-result} is sharp, cf. \cite[Proposition~5.2]{cecchini2022scalar}. For $m\geq 2$ even in the spin setting, Conjecture \ref{conj:long-neck-problem} holds by letting $\ell\to \pi/m$.

Following the ideas in Theorem \ref{thm:CZ-result}, and motivated by the scalar curvature estimate in \cite{goette2002scalar}, we obtain our first result on the long neck problem.
\begin{theorem}\label{thm:generalized-pointwise-long-neck-principle}
    Let $(N,g_N)$ be a compact $n$-dimensional Riemannian manifold with non-negative curvature operator $\mathcal{R}_N\geq 0$ on $\Lambda^2TN$, where $n\geq 2$ even. Let $(M,g_M)$ be a compact $m$-dimensional Riemannian spin manifold with non-empty boundary, where $m=n+4k$ for $k\in\mathbb{Z}_{\geq 0}$. Let $\Phi: M\to N$ be a smooth, area non-increasing, spin map that is locally constant near the boundary $\partial M$. Assume that
\begin{itemize}
    \item[\rm (i)] the Euler characteristic $\chi(N)$ of $N$ is non-vanishing;
    \item[\rm (ii)] the scalar curvature 
    \begin{equation}
     {\rm scal}_{M}\geq {\rm scal}_N\circ \Phi
    \end{equation}
    on $M$ and there exists a point $p_0\in U_{\frac{1}{2}}=\{p\in {\rm supp}({\rm d}\Phi)|\, \|\Lambda^2{\rm d}\Phi\|(p)<1/2\}$ such that $({\rm scal}_N \circ \Phi)(p_0)>0$;
    \item[\rm (iii)]  there exists $\ell\in \left(0,\pi\sqrt{\frac{m-1}{m\Xi}}\right)$ such that the mean curvature 
    \begin{equation}
    H \geq -\tan\left(\frac{\ell}{2}\sqrt{\frac{m\Xi}{m-1}}\right),\quad \text{where}\  \Xi=\sup\limits_{M} \{ {\rm scal}_N\circ \Phi  \}, 
     \end{equation}
    and the length of the neck ${\rm dist}_{g_M}({\rm supp}({\rm d}\Phi),\partial M)\geq \ell $.
\end{itemize}
Then ${\rm deg}_{\widehat{A}}(\Phi)=0$. 
\end{theorem}
For the definitions of curvature operators, area non-increasing spin maps and $\widehat{A}$-degree of a map see \cite{goette2002scalar}. Our theorem uses ${\rm deg}_{\widehat{A}}(\Phi)$ instead of ${\rm deg}(\Phi)$ when $\dim M=\dim N+4k$ for $k\in\mathbb{Z}_{\geq 0}$. If $\dim M=\dim N$ and $N$ is a unit sphere, then Theorem \ref{thm:generalized-pointwise-long-neck-principle} reduces to Theorem \ref{thm:CZ-result}.

It is well-known that a symmetric space $G/H$ of compact type has non-negative curvature operator. Moreover, ${\rm rank}\ G={\rm rank}\ H$ is equivalent to $\chi(G/H)\neq 0$ (cf.\cite{hopf1940satz}). In this case, Theorem \ref{thm:generalized-pointwise-long-neck-principle} reduces to an immediate corollary by \cite[Theorem~0.1]{goette2002scalar}.
\begin{corollary}
    Let $(N,g_N)$ be a compact $n$-dimensional Riemannian manifold with non-negative curvature operator $\mathcal{R}_N\geq 0$ on $\Lambda^2TN$, where $n\geq 2$ even, such that the universal covering of $N$ is homeomorphic to a symmetric space $G/H$ of compact type with ${\rm rank}\, G = {\rm rank}\, H$. Let $(M,g_M)$ be a compact $m$-dimensional Riemannian spin manifold with non-empty boundary, where $m=n+4k$ for $k\in\mathbb{Z}_{\geq 0}$. Let $\Phi: M\to N$ be a smooth, area non-increasing, spin map which is locally constant near the boundary $\partial M$. Assume that the scalar curvature ${\rm scal}_M\geq {\rm scal}_N\circ \Phi$ on $M$ and there exists a point $p_0 \in U_{\frac{1}{2}}=\{p\in {\rm supp}({\rm d}\Phi)|\, \|\Lambda^2{\rm d}\Phi\|(p)<1/2\}$ such that $({\rm scal}_N \circ \Phi)(p_0)>0$. Moreover, suppose that there exists $\ell\in \left(0,\pi\sqrt{\frac{m-1}{m\Xi}}\right)$ such that the mean curvature 
    \begin{equation}
     H \geq -\tan\left(\frac{\ell}{2}\sqrt{\frac{m\Xi}{m-1}}\right),\quad \text{where}\  \Xi=\sup\limits_{M} \{ {\rm scal}_N\circ \Phi \}, 
     \end{equation}
    and the length of the neck ${\rm dist}_{g_M}({\rm supp}({\rm d}\Phi),\partial M)\geq \ell$. Then ${\rm deg}_{\widehat{A}}(\Phi)=0$.
\end{corollary}

Now, we introduce another approach to the long neck problem. When $M$ is closed, the \textit{$c$-spectral constant} (cf. \cite{hirsch2023spectral}) is defined by
\begin{equation}\label{eq:c-spectral-constant}
\Lambda_c=\inf\left\{\int_{M}(|\nabla u|^2+c\ {\rm scal}_M u^2) {\rm vol}_M \Big|\, u\in H_0^1(M),\int_{M}u^2 {\rm vol}_{M}=1\right\}, c\in \mathbb{R}.
\end{equation}
Here $H_0^1(M)$ is the completion of $C_0^\infty(M)$ in the Sobolev $H^1$-norm. When $M$ is a compact manifold with boundary, $\Lambda_c$ is defined as the $c$-spectral constant of the interior $M^{\circ}$ of $M$ which coincides with the principal Dirichlet eigenvalue of the Schr\"{o}dinger operator $-\Delta+c\ {\rm scal}_M$, and the condition $\Lambda_c>0$ may be interpreted as a weak notion of positive scalar curvature if $c>0$. Note that if ${\rm scal}_M\geq \sigma>0$, then $\Lambda_c\geq c\sigma$. Different values of $c$ are used in variant geometric contexts: the search for black holes, the Yamabe problem, minimal surfaces, and Ricci flow with surgery, etc. See Li-Mantoulidis \cite{li2023metrics} for an extended discussion.

In the recent work \cite[Theorem~1.2]{hirsch2023spectral}, the authors estimated the widths of $\widehat{A}$-overtorical bands with respect to $c$-spectral $\Lambda_c$ under the spin assumption by modifying the arguments of spinorial Callias operators in \cite{cecchini2022scalar}. Inspired by this, we employ analogous methods to generalize the long neck principle in this paper. We state here our second result---a ``spectral long neck principle'' below.
\begin{theorem}\label{thm:main}
    Let $(N,g_N)$ be a compact $n$-dimensional Riemannian manifold with non-negative curvature operator $\mathcal{R}_N\geq 0$ on $\Lambda^2TN$, where $n\geq 2$ even. Let $(M,g_M)$ be a compact $m$-dimensional Riemannian spin manifold with non-empty boundary, where $m=n+4k$ for $k\in\mathbb{Z}_{\geq 0}$. Let $\Phi: M\to N$ be a smooth spin map which is locally constant near the boundary $\partial M$. Take $c>\frac{m-1}{4m}$ such that $\Lambda_c>0$. Assume that
\begin{itemize}
    \item[\rm (i)] the Euler characteristic $\chi(N)$ of $N$ is non-vanishing;
    \item[\rm (ii)] the scalar curvature 
    \begin{equation}\label{eq:scalar-comparison}
     {\rm scal}_{M} \geq ({\rm scal}_N\circ \Phi) \cdot \|\Lambda^2{\rm d}\Phi\|
    \end{equation}
    on ${\rm supp}({\rm d}\Phi)$;  
    \item[\rm (iii)] the length of the neck
    \begin{equation}\label{eq:spectral-bound-of-neck}
       {\rm dist}_{g_M}({\rm supp}({\rm d}\Phi),\partial M) > \pi\sqrt{\frac{c}{\Lambda_c}\left(\frac{(4c-1)m+2-4c}{(4c-1)m+1}\right)}.
    \end{equation}
\end{itemize}
Then ${\rm deg}_{\widehat{A}}(\Phi)=0$.
\end{theorem}
This result generalizes Theorem \ref{thm:cecchini2020longneck} by omitting the assumption $\|\Lambda^2{\rm d}\Phi\|_{\infty}\leq 1$ and using area contraction constant $\|\Lambda^2{\rm d}\Phi\|$ on ${\rm supp}({\rm d}\Phi)$ in the scalar curvature comparison inequality \eqref{eq:scalar-comparison}. Note that this is a significant improvement since at points where $\|\Lambda^2{\rm d}\Phi\|$ vanishes, \eqref{eq:scalar-comparison} becomes ${\rm scal}_{M}\geq 0$ on ${\rm supp}({\rm d}\Phi)$. Moreover, when $M$ is diffeomorphic to $N$ and $\Phi$ is the identity on $M$, the inequality \eqref{eq:scalar-comparison} is equivalent to (compare \cite{listing2010scalar})
\begin{equation}
\begin{aligned}
  {\rm scal}_{M}\geq 0\quad & \text{on}\quad {\rm supp}({\rm d}\Phi)\\
  \text{and}\quad ({\rm scal}_{M})^2\cdot g_M\geq ({\rm scal}_{N})^2\cdot g_N\quad & \text{on}\quad \Lambda^2T({\rm supp}({\rm d}\Phi)).
\end{aligned}
\end{equation}
For $m\geq 2$ even, if we assume the pointwise conditions ${\rm scal}_M \geq \sigma>0$ and the same neck length inequality as in \eqref{eq:cecchini2020necklength} instead of the above spectral assumptions, then it again yields that the $\widehat{A}$-degree of comparison map vanishes, from our spectral result by observing that $\Lambda_c\geq c \sigma$ and letting $c\to \infty$.

According to \cite[Theorem~0.1]{goette2002scalar}, we have the following direct consequences of Theorem \ref{thm:main}.
\begin{corollary}
Let $(N,g_N)$ be a compact $n$-dimensional Riemannian manifold with non-negative curvature operator $\mathcal{R}_N\geq 0$ on $\Lambda^2TN$, where $n\geq 2$ even, such that the universal covering of $N$ is homeomorphic to a symmetric space $G/H$ of compact type with ${\rm rank}\ G={\rm rank}\ H$. Let $(M,g_M)$ be a compact $m$-dimensional Riemannian spin manifold with non-empty boundary, where $m=n+4k$ for $k\in\mathbb{Z}_{\geq 0}$. Let $\Phi: M\to N$ be a smooth, area non-increasing, spin map which is locally constant near the boundary $\partial M$. Take $c>\frac{m-1}{4m}$ such that $\Lambda_c>0$. Assume that the scalar curvature ${\rm scal}_{M}\geq {\rm scal}_{N}\circ\Phi$ on ${\rm supp}({\rm d}\Phi)$ and the length of the neck
    \begin{equation}
       {\rm dist}_{g_M}({\rm supp}({\rm d}\Phi),\partial M) > \pi\sqrt{\frac{c}{\Lambda_c}\left(\frac{(4c-1)m+2-4c}{(4c-1)m+1}\right)}.
    \end{equation}
Then ${\rm deg}_{\widehat{A}}(\Phi)=0$.  
\end{corollary}
\begin{corollary}\label{cor:spectral-long-neck-for-sphere}
Let $(M,g_M)$ be a compact $m$-dimensional Riemannian spin manifold with boundary, where $m\geq 2$ even, and let $\Phi:M\to S^m$ be a smooth area non-increasing map. Take $c>\frac{m-1}{4m}$ such that $\Lambda_c>0$. Assume that the scalar curvature ${\rm scal}_{M}\geq m(m-1)$ on ${\rm supp}({\rm d}\Phi)$ and the length of the neck
    \begin{equation}
       {\rm dist}_{g_M}({\rm supp}({\rm d}\Phi),\partial M) > \pi\sqrt{\frac{c}{\Lambda_c}\left(\frac{(4c-1)m+2-4c}{(4c-1)m+1}\right)}.
    \end{equation}
Then ${\rm deg}(\Phi)=0$. 
\end{corollary}
Corollary \ref{cor:spectral-long-neck-for-sphere} generalizes Theorem \ref{thm:cecchini2020longneck}, by noting that $\Lambda_c\geq c\sigma$ under the pointwise bound ${\rm scal_M}\geq \sigma>0$ and then letting $c\to \infty$. In fact, the pointwise version of Corollary \ref{cor:spectral-long-neck-for-sphere} also holds in the odd-dimensional case. This can be seen by considering the product manifold $M\times S^1$.

Let us now consider another conjecture proposed by Gromov, which is related to the long neck problem. Specifically,
\begin{conjecture}[{\cite[11.12, Conjecture D']{gromov2018metric}}, Geodesic Collar Neighborhood Problem]\label{conj:geodesic-collar-neighborhood}
Let $W$ be a closed $m$-dimensional manifold such that $W$ minus a point does not admit a complete metric of positive scalar curvature. Let $M$ be the manifold with boundary obtained from $W$ by removing a small $m$-dimensional open ball. Let $g_M$ be a Riemannian metric on $M$ such that ${\rm scal}_{M}\geq \sigma>0$. Then there exists a constant $c>0$ such that if there exists an open geodesic collar neighborhood $\mathcal{N}$ of width $\rho$, then
\begin{equation}
\rho\leq\frac{c}{\sqrt{\sigma}}.
\end{equation} 
\end{conjecture}
Here observe that $M$ is a manifold with boundary $\partial M=S^{n-1}$ which is homeomorphic to $S^{n-1}\times [0,1]$. For $\rho>0$ small enough, an open geodesic collar neighborhood $\mathcal{N}$ of width $\rho$ is the $\rho$-neighborhood of $\partial M$ with respect to $g_M$.

This conjecture was solved and extended to a higher version \cite[Theorem~B, Theorem~C]{cecchini2020along} by Cecchini. Subsequently, Cecchini and Zeidler \cite{cecchini2022scalar} refined the upper bound of Conjecture \ref{conj:geodesic-collar-neighborhood} in the case of an even-dimensional spin manifold whose double has infinite $\widehat{A}$-area (for precise definition and more examples of the $\widehat{A}$-area see \cite[Definition~1.6]{cecchini2022scalar}) using information from the mean curvature of the boundary.
\begin{theorem}[{\cite[Theorem~1.7]{cecchini2022scalar}}]\label{thm:CZ-geodesic-collar-neighborhood}
\rm Let $(M,g_M)$ be a compact $m$-dimensional Riemannian spin manifold with boundary such that the double of $M$ has infinite $\widehat{A}$-area. Suppose that ${\rm scal}_{M}>0$ and that there exist positive constants $\kappa$ and $\ell$, with $0<\ell<\pi/(\sqrt{\kappa} m)$, such that the mean curvature
\[
    H \geq -\sqrt{\kappa}\tan\left(\frac{\sqrt{\kappa}m \ell}{2}\right).
\]
Then the boundary $\partial M$ admits no open geodesic collar neighborhood $\mathcal{N}\subset M$ of width strictly greater than $\ell$ such that ${\rm scal}_{M}\geq \kappa m(m-1)$ on $\mathcal{N}$. 
\end{theorem}
The width estimate in Theorem \ref{thm:CZ-geodesic-collar-neighborhood} is also sharp, cf. \cite[Remark~1.13]{cecchini2022scalar}. In particular, this result implies that a manifold with boundary whose double has infinite $\widehat{A}$-area cannot carry any metric of positive scalar curvature and mean convex boundary, compare \cite[Theorem~19]{baer2020boundary}. Here we establish our third theorem, a spectral width inequality of geodesic collar neighborhood.
\begin{theorem}\label{thm:spectral-geodesic-collar-neighborhood}
Let $(M,g_M)$ be a compact $m$-dimensional Riemannian spin manifold with boundary such that the double of $M$ has infinite $\widehat{A}$-area. Take $c>\frac{m-1}{4m}$ such that $\Lambda_c>0$. Then the width of an open geodesic collar neighborhood $\mathcal{N}$ of the boundary $\partial M$ 
\begin{equation}
{\rm width}(\mathcal{N}) \leq  \pi\sqrt{\frac{c}{\Lambda_c}\left(\frac{(4c-1)m+2-4c}{(4c-1)m+1}\right)}.
\end{equation}
\end{theorem}
If we assume the pointwise bound ${\rm scal}_M \geq \sigma>0$ and note that $\Lambda_c\geq c \sigma$, then applying Theorem \ref{thm:spectral-geodesic-collar-neighborhood} while letting $c\to \infty$ we get the pointwise width estimate ${\rm width}(\mathcal{N})\leq \pi\sqrt{\frac{m-1}{m\sigma}}$.

Throughout this paper all manifolds are assumed to be smooth, oriented and connected. 

The paper is organized as follows. In Section \ref{sec:preparations}, we provide some necessary notions and basic facts of technical preparations. In Section \ref{sec:proof-of-theorems}, first of all, we prove Theorem \ref{thm:generalized-pointwise-long-neck-principle} using a crucial lemma---Lemma \ref{lem:relative-GL-pair} of the construction of a relative Gromov-Lawson pair on a compact manifold with boundary, relative to a background manifold. Next, we prove Theorem \ref{thm:main} by Lemma \ref{lem:relative-GL-pair} in the spectral setting. Finally, we prove Theorem \ref{thm:spectral-geodesic-collar-neighborhood} by \cite[Lemma~6.9]{cecchini2022scalar} together with the similar arguments and estimates of spinors of Callias operators in the proof of Theorem \ref{thm:main}.

\section{Preliminaries: Callias Operator Approach}\label{sec:preparations}
Before proving our results, we will introduce the requisite machinery and notation.
Let $M$ be a compact even-dimensional Riemannian spin manifold with boundary. Let $\mathcal{S}_{M}\to M$ denote the associated complex spinor bundle, equipped with the natural spin connection induced by the Levi-Civita connection. Let $\mathcal{E},\mathcal{F}\to M$ be a pair of Hermitian bundles with metric connections, and consider the $\mathbb{Z}/2$-graded Dirac bundle 
\begin{equation}
    S=\mathcal{S}_{M}\widehat{\otimes} \footnote{Here the symbol $\widehat{\otimes}$ denotes the $\mathbb{Z}/2$-graded tensor product, cf. \cite{quillen1985superconnections}. If a vector bundle is ungraded, we understand it as $\mathbb{Z}/2$-graded by taking its odd part as zero.} (\mathcal{E}\oplus \mathcal{F}^{\rm op}),
\end{equation}
where the grading $S=S^+\oplus S^-$ is given by
\[
    S^+:=(\mathcal{S}_{M}^+\otimes \mathcal{E})\oplus (\mathcal{S}_{M}^- \otimes \mathcal{F})\quad\text{and}\quad S^-:=(\mathcal{S}_{M}^+\otimes \mathcal{F})\oplus (\mathcal{S}_{M}^- \otimes \mathcal{E}).
\]
This bundle carries a natural involution $\sigma$ given by
\begin{equation}
\sigma:={\rm id}_{\mathcal{S}_{M}}\widehat{\otimes}\begin{pmatrix}
0& \daleth^*\\
\daleth & 0
\end{pmatrix}
:\left. S\right|_{M\setminus K}\to \left. S\right|_{M\setminus K},
\end{equation}
where $K$ is a compact subset of the interior $M^{\circ}$, ${\rm id}_{\mathcal{S}_M}$ means the identity on $\mathcal{S}_{M}$ and $\daleth:\left.\mathcal{E}\right|_{M\setminus K}\to \left.\mathcal{F}\right|_{M\setminus K}$ denotes a parallel unitary bundle isomorphism which extends to a smooth bundle map on a neighborhood of $\overline{M\setminus K}$. Such bundles $\mathcal{E},\mathcal{F}$ form a \textit{Gromov-Lawson pair} with \textit{support} $K$, cf. \cite[Example~2.5]{cecchini2022scalar}.

Let $\mathcal{D}_{\mathcal{E}}$ and $\mathcal{D}_{\mathcal{F}}$ be the twisted Dirac operators on $M$ respectively with the bundles $\mathcal{E}$ and $\mathcal{F}$, then we may form the corresponding Dirac operator $\mathcal{D}$ on $S$, see \cite[Example~2.5]{cecchini2022scalar}. Given a Lipschitz function $f$ on $M$, the associated \textit{Callias operator} is defined by
\begin{equation}\label{eq:defn-of-Callias-operator}
\mathcal{B}_{f}u= \mathcal{D}u+f\sigma u.
\end{equation} 
By the $\mathbb{Z}/2$-graded structure on $S$, the differential operator $\mathcal{B}_f$ can be decomposed as $\mathcal{B}_f=\mathcal{B}_f^+\oplus\mathcal{B}_f^-$ where $\mathcal{B}_f^{\pm}: C^\infty(M,S^{\pm})\to C^\infty(M,S^{\mp})$, cf. \cite[Section~3]{cecchini2022scalar}.

We now consider an elliptic boundary value problem 
\begin{equation}\label{eq:boundary-value-problem}
\mathcal{B}_fu=0\quad \text{in}\ M,\qquad s\nu\cdot\sigma u=u\quad \text{on}\ \partial M,
\end{equation}
where $s:\partial M\to \{\pm 1\}$ denotes the choice of signs and $\nu\in C^\infty(\partial M,\left. TM\right|_{\partial M})$ denotes the outward pointing unit normal vector field. Let $\mathcal{B}_{f,s}$ denote the operator $\mathcal{B}_f$ on the domain
\begin{equation}\label{eq:domain}
    C_{\sigma,s}^\infty(M,S):=\left\{u\in C_0^\infty(M,S) \Big| \,  s\nu\cdot\sigma\left.u\right|_{\partial M}=\left.u\right|_{\partial M}\right\}.
\end{equation}
Thanks to \cite[Theorem~3.4]{cecchini2022scalar}, $\mathcal{B}_{f,s}$ is self-adjoint and Fredholm.
Since $s\nu\cdot\sigma$ is even with respect to the grading on $S$, the operator $\mathcal{B}_{f,s}$ can also be decomposed as $\mathcal{B}_{f,s}=\mathcal{B}_{f,s}^+\oplus\mathcal{B}_{f,s}^-$ where $\mathcal{B}_{f,s}^{\pm}$ are adjoint to each other. Therefore, 
\begin{equation}\label{eq:defn-of-index}
{\rm ind}(\mathcal{B}_{f,s})=\dim(\ker(\mathcal{B}_{f,s}^{+}))-\dim(\ker(\mathcal{B}_{f,s}^{-})).
\end{equation}

Given a Gromov-Lawson pair $(\mathcal{E},\mathcal{F})$ on $M$, let ${\rm d}M:=M\cup_{\partial M}M^-$ be the double of $M$, where $M^-$ denotes the manifold $M$ with opposite orientation. Note that ${\rm d}M$ is a closed manifold carrying a spin structure induced by the spin structure of $M$. Let $V(\mathcal{E},\mathcal{F})\to {\rm d}M$ be a bundle on ${\rm d}M$ which outside a small collar neighborhood coincides with $\mathcal{E}$ over $M$ and with $\mathcal{F}$ over $M^{-}$ by the bundle isomorphism $\daleth$ in a Gromov-Lawson pair. Now, we define the \textit{relative index} of the pair $(\mathcal{E},\mathcal{F})$
\begin{equation}
{\rm indrel}(M;\mathcal{E},\mathcal{F}):={\rm ind}(\mathcal{D}_{{\rm d}M,V(\mathcal{E},\mathcal{F})}).
\end{equation}
From \cite[Corollary~3.9]{cecchini2022scalar}, we have
\begin{equation}\label{eq:GGL}
{\rm ind}(\mathcal{B}_{f,-1})={\rm indrel}(M;\mathcal{E},\mathcal{F}).
\end{equation}

A smooth map $\Phi: M\to N$ between two oriented manifolds $M$ and $N$ is a \textit{spin map}, cf. \cite[Section~1]{goette2002scalar}, if it is compatible with the second Stiefel-Whitney classes, i.e., $w_2(TM)=\Phi^* w_2(TN)$. In particular, this condition implies that the bundle $TM\oplus \Phi^*TN$ admits a spin structure.
Now, if $M$ is a compact even-dimensional Riemannian manifold with boundary and $N$ is a compact Riemannian manifold, and $\Phi,\Psi$ are two spin maps from $M$ to $N$, then we think of $N$ as the background space and $M$ as the manifold whose relative Gromov-Lawson pair we want to define, relative to $N$. We denote by $\mathcal{S}_M$ and $\mathcal{S}_N$ the locally defined spinor bundles of $M$ and $N$ respectively. We use $\Phi^*\mathcal{S}_{N}, \Psi^*\mathcal{S}_N$ to denote the pair of Hermitian bundles over $M$, which are the pull-back bundles of the complex spinor bundle $\mathcal{S}_N$ over $N$ via the maps $\Phi,\Psi$ respectively. Note that $\Phi^*\mathcal{S}_{N}$ and $\Psi^*\mathcal{S}_N$ carry the metric connections induced by the spin connection on $\mathcal{S}_N$ using the pull-back maps respectively. Note that $\mathcal{S}_M \widehat{\otimes} \Phi^*\mathcal{S}_N$ and $\mathcal{S}_M \widehat{\otimes} \Psi^*\mathcal{S}_N$ are globally defined although $\mathcal{S}_M$ and $\mathcal{S}_N$ are defined locally. In this case, set
\begin{equation}
S=\mathcal{S}_M\widehat{\otimes}\big(\Phi^*\mathcal{S}_N\oplus (\Psi^*\mathcal{S}_N)^{\rm op} \big). 
\end{equation}
We can retain the above construction of the Gromov-Lawson pair by restricting to the bundles $\Phi^*\mathcal{S}_N$ and $\Psi^*\mathcal{S}_N$. Such a pair of bundles $\Phi^*\mathcal{S}_M,\Psi^*\mathcal{S}_N$ is called a \textit{relative Gromov-Lawson pair} with \textit{support} $K$.

Let $\mathcal{D}_{\Phi^*\mathcal{S}_N}$ and $\mathcal{D}_{\Psi^*\mathcal{S}_N}$ be the twisted Dirac operators over $M$ respectively with the bundles $\Phi^*\mathcal{S}_N$ and $\Psi^*\mathcal{S}_N$. We form the corresponding Dirac operator $\mathcal{D}$ on $S$, compare \cite[Example~2.5]{cecchini2022scalar}. Given a Lipschitz function $f$ on $M$, the \textit{Callias operator} $\mathcal{B}_f$ on $S$ is defined as in \eqref{eq:defn-of-Callias-operator}. By the $\mathbb{Z}/2$-graded structure on $S$, the operator $\mathcal{B}_f$ can be decomposed as $\mathcal{B}_f=\mathcal{B}_f^+\oplus\mathcal{B}_f^-$ where $\mathcal{B}_f^{\pm}: C^\infty(M,S^{\pm})\to C^\infty(M,S^{\mp})$. We also consider an elliptic boundary value problem as in \eqref{eq:boundary-value-problem}. Let $\mathcal{B}_{f,s}$ denote the operator $\mathcal{B}_f$ on the domain $C_{\sigma,s}^\infty(M,S)$ as in \eqref{eq:domain}. As in \cite[Theorem~3.4]{cecchini2022scalar}, $\mathcal{B}_{f,s}$ is self-adjoint and Fredholm.
Since $s\nu\cdot\sigma$ is even with respect to the grading on $S$, the operator $\mathcal{B}_{f,s}$ can also be decomposed as $\mathcal{B}_{f,s}=\mathcal{B}_{f,s}^+\oplus\mathcal{B}_{f,s}^-$ where $\mathcal{B}_{f,s}^{\pm}$ are adjoint to each other. Therefore, the index of $\mathcal{B}_{f,s}$ is defined as in \eqref{eq:defn-of-index}.

Given a relative Gromov-Lawson pair $(\Phi^*\mathcal{S}_N,\Psi^*\mathcal{S}_N)$  over $M$ associated to the spin maps $\Phi,\Psi:M\to N$, note that the bundle $\mathcal{S}_M \widehat{\otimes} V(\Phi^*\mathcal{S}_N,\Psi^*\mathcal{S}_N)$ over ${\rm d}M$ which outside a small collar neighborhood coincides with $\mathcal{S}_M \widehat{\otimes} \Phi^*\mathcal{S}_N$ over $M$ and with $\mathcal{S}_M \widehat{\otimes} \Psi^*\mathcal{S}_N$ over $M^{-}$ defined using the tensor product ${\rm id}_{\mathcal{S}_M}\widehat{\otimes} \daleth$ of the identity on $\mathcal{S}_M$, where $\daleth$ is the bundle isomorphism in the relative Gromov-Lawson pair. In this case, since the bundles $\mathcal{S}_M \widehat{\otimes} \Phi^*\mathcal{S}_N,\mathcal{S}_M \widehat{\otimes} \Psi^*\mathcal{S}_N$ are globally defined over $M$ and $M^-$ respectively, the bundle $\mathcal{S}_M \widehat{\otimes} V(\Phi^*\mathcal{S}_N,\Psi^*\mathcal{S}_N)$ is also globally defined on ${\rm d}M$. Then we define the \textit{relative index} of the relative Gromov-Lawson pair $(\Phi^*\mathcal{S}_N,\Psi^*\mathcal{S}_N)$
\begin{equation}
{\rm indrel}(M;\Phi^*\mathcal{S}_N,\Psi^*\mathcal{S}_N):={\rm ind}(\mathcal{D}_{{\rm d}M,V(\Phi^*\mathcal{S}_N,\Psi^*\mathcal{S}_N)}).
\end{equation}
\begin{lemma}\label{lem:index-coincide-relative-index}
For any choice of the potential $f$, we have
\begin{equation}\label{eq:index-coincide-relative-index}
{\rm ind}(\mathcal{B}_{f,-1})={\rm indrel}(M;\Phi^*\mathcal{S}_N,\Psi^*\mathcal{S}_N).
\end{equation}
\end{lemma}
\begin{remark}
In \cite[Corollary~3.9]{cecchini2022scalar}, Cecchini and Zeidler have dealt explicitly with the case of the Gromov-Lawson pair. The methods and results in \cite[Corollary~3.9]{cecchini2022scalar} holds for the relative Gromov-Lawson pair without change. 
\end{remark}
For the sake of completeness, we include a proof here.
\begin{proof}[Proof of Lemma \ref{lem:index-coincide-relative-index}]
Let $\mathcal{E}'=V(\Phi^*\mathcal{S}_N,\Psi^*\mathcal{S}_N)$. We extend the bundle $\Psi^*\mathcal{S}_N$ to a bundle $\mathcal{F}'$ on ${\rm d}M$ with metric connection such that $\left.\mathcal{F}'\right|_{M^{-}}=\Psi^*\mathcal{S}_N$. Consider the $\mathbb{Z}/2$-graded Dirac bundle $W':=\mathcal{S}_{{\rm d}M}\widehat{\otimes}(\mathcal{E}'\oplus (\mathcal{F}')^{\rm op})$ with associated Dirac operator $\mathcal{D}'$. Note that
\begin{equation}\label{eq:index-dirac-on-double}
{\rm ind}(\mathcal{D}')={\rm indrel}(M;\Phi^*\mathcal{S}_N,\Psi^*\mathcal{S}_N)
\end{equation}
because the index of the Dirac operator on $\mathcal{S}_{{\rm d}M} \widehat{\otimes} \mathcal{F}'$ vanishes. Observe also that $\mathcal{D}'=\mathcal{B}_0^{\mathcal{E}',\mathcal{F}'}$, where $\mathcal{B}_0^{\mathcal{E}',\mathcal{F}'}$ is the Callias operator associated to the pair $(\mathcal{E}',\mathcal{F}')$. Cut ${\rm d}M$ open along $\partial M$ as in \cite[Theorem~3.6]{cecchini2022scalar}. By pulling back all data, we obtain the operator $\mathcal{B}_{0,-1}^{\Phi^*\mathcal{S}_N,\Psi^*\mathcal{S}_N}$ on $M$ and $\mathcal{B}_{0,-1}^{\Psi^*\mathcal{S}_N,\Psi^*\mathcal{S}_N}$ on $M^{-}$. By \cite[Lemma~3.8]{cecchini2022scalar}, ${\rm ind}(\mathcal{B}_{0,-1}^{\Psi^*\mathcal{S}_N,\Psi^*\mathcal{S}_N})=0$. By \cite[Lemma~3.5]{cecchini2022scalar}, ${\rm ind}(\mathcal{B}_{f,-1}^{\Phi^*\mathcal{S}_N,\Psi^*\mathcal{S}_N})={\rm ind}(\mathcal{B}_{0,-1}^{\Phi^*\mathcal{S}_N,\Psi^*\mathcal{S}_N})$. Therefore, the lemma follows from \eqref{eq:index-dirac-on-double} and \cite[Theorem~3.6]{cecchini2022scalar}. 
\end{proof}

\section{Proofs of Main Theorems}\label{sec:proof-of-theorems}
The purpose of this section is to establish our main results: Theorem \ref{thm:generalized-pointwise-long-neck-principle}, Theorem \ref{thm:main}, and Theorem \ref{thm:spectral-geodesic-collar-neighborhood} using the technical preparations from Section \ref{sec:preparations}.

\subsection{Proof of Theorem \ref{thm:generalized-pointwise-long-neck-principle}}\label{subsec:generalized-pointwise-long-neck}
Let us state the following constructive lemma of a relative Gromov-Lawson pair, which plays an important role in the proof of Theorem \ref{thm:generalized-pointwise-long-neck-principle}.
\begin{lemma}\label{lem:relative-GL-pair}
Let $(N,g_N)$ be a compact $n$-dimensional Riemannian manifold with non-negative curvature operator $\mathcal{R}_N\geq 0$ on $\Lambda^2TN$ and non-vanishing Euler characteristic, where $n\geq 2$ even. Let $(M,g_M)$ be a compact $m$-dimensional Riemannian spin manifold with non-empty boundary, where $m=n+4k$ for $k\in\mathbb{Z}_{\geq 0}$. Let $\Phi: M\to N$ be a spin map which is locally constant near the boundary $\partial M$ and of non-zero $\widehat{A}$-degree. If the length of the neck ${\rm dist}_{g_M}({\rm supp}({\rm d}\Phi),\partial M)\geq \ell$ for some number $\ell$, then there exists a relative Gromov-Lawson pair $(\mathcal{\mathcal{E}},\mathcal{F})$ such that 
\begin{itemize}
\item[\rm (i)] $\mathscr{R}_p^{\mathcal{E}\oplus \mathcal{F}}\geq -({\rm scal}_{N}\circ \Phi)\cdot \|\Lambda^2{\rm d}\Phi\|(p)/4$ for all $p\in M$;
\item[\rm (ii)] $(\mathcal{E},\mathcal{F})$ has support $K:=\{p\in M|{\rm dist}_{g_M}(p,\partial M) \geq \ell\}\supseteq {\rm supp}({\rm d}\Phi)$;
\item[\rm (iii)] ${\rm indrel}(M;\mathcal{E},\mathcal{F})\neq 0$.
\end{itemize}
\end{lemma}
\begin{proof}
We follow the approach of \cite[Lemma~5.1]{cecchini2022scalar}. Namely, we construct a spin map $\Psi: M\to N$. Note that $\Psi$ is a spin map because $\Psi$ coincides with $\Phi$ near $\partial M$, $\Psi$ is constant on the other part and $M$ is spin.
Then we pass to the relative Gromov-Lawson pair construction associated to the spin maps $\Phi, \Psi$. The main difference in our case is that we obtain the estimate of curvature term $\mathscr{R}^{\Phi^*\mathcal{S}_N\oplus \Psi^*\mathcal{S}_N}$ and the obstruction to the relative index of the pair $(\Phi^*\mathcal{S}_N, \Psi^*\mathcal{S}_N)$ via Goette-Semmelmann's argument \cite[Equation~1.11, Theorem~2.4]{goette2002scalar}. Details are left to the reader. 

\end{proof}
\begin{proof}[Proof of Theorem \ref{thm:generalized-pointwise-long-neck-principle}]
Argue by contradiction together with Lemma \ref{lem:relative-GL-pair} and then revise the proof of \cite[Theorem~1.4]{cecchini2022scalar} verbatim in our setting. Details are left to the reader.  
\end{proof}

\subsection{Proof of Theorem \ref{thm:main}}\label{subsec:Proof-Main-Theorem-1}
In this subsection, we prove our spectral long neck principle, Theorem \ref{thm:main}.
\begin{proof}[Proof of Theorem \ref{thm:main}]
Let $f$ be a Lipschitz function on $M$ such that $f$ is non-negative on all $M$ and tends to infinity on $\partial M$, to be constructed later. Let $(\mathcal{E},\mathcal{F})$ be a relative Gromov-Lawson pair, to be constructed later as well. We denote by
${\rm II}_{\partial M}$ the second fundamental form of $\partial M$ with respect to $\nu$, and by $H=\frac{1}{m-1}{\rm tr\, II}_{\partial M}$ the mean curvature of $\partial M$. Set $\beta=\frac{m}{m-1}-\frac{1}{4c}>0$ and $\beta_1=\frac{1}{m(m-1)}-\frac{1}{\beta(m-1)^2}$. If we have a non-trivial element $u\in\ker(\mathcal{B}_{f,-1})$, according to \cite[Equation~4.4]{cecchini2022scalar}, we have
\begin{equation}
\begin{aligned}
0=&\int_{M} \frac{m}{m-1}\left(|\mathcal{P}u|^2+\langle u,\frac{{\rm scal}_M}{4}u+\mathscr{R}^{\mathcal{E}\oplus\mathcal{F}}u\rangle\right) {\rm vol}_{M} \\
&\qquad \qquad + \int_{M} \langle u, f^2u+\nabla f\cdot \sigma u\rangle {\rm vol}_{M}+\int_{\partial M}(f+\frac{m}{2}H)|u|^2 {\rm vol}_{\partial M},
\end{aligned}
\end{equation}
where $\mathcal{P}$ is the Penrose operator, and $\mathscr{R}^{\mathcal{E}\oplus\mathcal{F}}$ is the curvature term acting on sections of $S$ in the Bochner-Lichnerowicz-Weitzenb\"{o}ck formula, compare \cite[Equation~2.18]{cecchini2022scalar}.

Using \cite[Equation~3.10]{hirsch2023spectral}, we have
\begin{equation}\label{eq:consequence-of-BLW}
\begin{aligned}
-\int_{\partial M}(f+\frac{m}{2}H) & |u|^2 {\rm vol}_{\partial M} \geq \int_{M}\left[\frac{m}{m-1}\left(\frac{1}{4c}\big|\nabla|u|\big|^2+\beta_1f^2|u|^2\right.\right.\\
&\left.\left. +\langle u,\frac{{\rm scal}_M}{4}u+\mathscr{R}^{\mathcal{E}\oplus\mathcal{F}}u\rangle \right)+\langle u, f^2u+\nabla f\cdot\sigma u\rangle\right] {\rm vol}_{M},
\end{aligned}
\end{equation}
Multiplying the following relation (see e.g. \cite[Equation~3.12]{hirsch2023spectral}) by $\frac{m}{m-1}\beta_1$ 
\begin{equation}\label{eq:boundary-identity}
\begin{aligned}
\int_{\partial M}\langle \nu\cdot u, f\sigma u\rangle {\rm vol}_{\partial M} 
&=-\int_{M}\left(2f^2|u|^2+\langle u,\nabla f\cdot \sigma u\rangle\right) {\rm vol}_{M},
\end{aligned}
\end{equation}
and then summing the result with \eqref{eq:consequence-of-BLW}, we have
\begin{equation}\label{eq:3.13}
\begin{aligned}
&-\int_{\partial M}\left[\left(1-\frac{m\beta_1}{m-1}\right)f+\frac{m}{2}H\right]|u|^2 {\rm vol}_{\partial M}\\
\geq & \int_{M}\frac{m}{m-1}\left(\frac{1}{4c}\big|\nabla|u|\big|^2+\langle u, \frac{{\rm scal}_M}{4}u+\mathscr{R}^{\mathcal{E}\oplus\mathcal{F}}u\rangle\right) {\rm vol}_{M}\\
&\qquad +\int_{M}\left\langle u, \underbrace{\left(1-\frac{m\beta_1}{m-1}\right)}_{\beta_2}f^2u+\left(1-\frac{m\beta_1}{m-1}\right) \nabla f\cdot\sigma u\right\rangle {\rm vol}_{M}.
\end{aligned}
\end{equation}
Clearly, $\beta_1<\frac{1}{m(m-1)}$, thus $\beta_2=1-\frac{m\beta_1}{m-1}>1-\frac{1}{(m-1)^2}>0$. Hence the boundary term of \eqref{eq:3.13} is non-positive, provided that $f$ is sufficiently large on $\partial M$.

Assume, by contradiction, that ${\rm deg}_{\widehat{A}}(\Phi)\neq 0$. Moreover, restating the hypotheses of the theorem, we have 
\begin{itemize}
\item[(iv)] 
 ${\rm scal}_{M}\geq ({\rm scal}_N\circ \Phi)\cdot \|\Lambda^2{\rm d}\Phi\|$ on ${\rm supp}({\rm d}\Phi)$; 
\item[(v)] there exists an $\varepsilon>0$ such that 
\[
{\rm dist}_{g_M}({\rm supp}({\rm d}\Phi),\partial M) \geq \omega:=\pi\sqrt{\frac{\beta_2c(m-1)}{m\Lambda_c}}+\varepsilon=\pi\sqrt{\frac{c}{\Lambda_c}\left(\frac{(4c-1)m+2-4c}{(4c-1)m+1}\right)}+\varepsilon.
\]
\end{itemize}
Then we choose a relative Gromov-Lawson pair $(\mathcal{E},\mathcal{F})$ satisfying the conditions (i) to (iii) from Lemma \ref{lem:relative-GL-pair}.

Next, we define a sequence of bounded Lipschitz functions $f_j$ on $M$, which satisfy a certain differential inequality and have the property that $f_j\to +\infty$ on $\partial M$ as $j\to +\infty$, in the following way as in \cite[Section~3]{hirsch2023spectral}. Let $r(p)={\rm dist}_{g_M}(p,\partial M)$, and for each $j$ let 
\begin{equation}\label{eq:defn-potential}
\begin{aligned}
f_j(p)=\begin{cases}
\frac{\pi}{2\omega}\cot\left(\frac{\pi}{2\omega}r(p)+\frac{1}{j}\right)\quad &\text{if}\ r(p)\leq \frac{2\omega}{\pi}(\frac{\pi}{2}-\frac{1}{j})\\
0\quad &\text{otherwise}.
\end{cases}
\end{aligned}
\end{equation}
Since ${\rm dist}_{g_M}({\rm supp}({\rm d}\Phi),\partial M)\geq \omega$, we have that $\left.f_j\right|_{{\rm supp}({\rm d}\Phi)}=0$. Fix a compact subset $\Omega\subset M^{\circ}$ such that for all sufficiently large $j$ we have
\begin{subequations}
\begin{align}
\frac{3}{2}f_j^2-|\nabla f_j|\geq 1\quad &\text{on}\ M\setminus\Omega,\label{eq:characteristic-1}\\
\beta_2f_j^2-\beta_2|\nabla f_j|+\frac{m\Lambda_c}{4c(m-1)}\geq C_{\varepsilon}\quad &\text{on}\ M, \label{eq:characteristic-2}
\end{align}
\end{subequations}
where $C_{\varepsilon}>0$ depends on $\varepsilon, m, c$ and $\Lambda_c$. Equation \eqref{eq:index-coincide-relative-index} together with (iii) in Lemma \ref{lem:relative-GL-pair} implies that the corresponding Callias operator subject to the sign $s=-1$ satisfies
\begin{equation}\label{eq:non-vanishing-of-index}
{\rm ind}(\mathcal{B}_{f_j,-1})={\rm indrel}(M;\mathcal{E},\mathcal{F})\neq 0.
\end{equation}
In particular, we may obtain a non-trivial element $u_j\in\ker(\mathcal{B}_{f_j,-1})$ for each $j$. 
Similarly to \cite{hirsch2023spectral}, from equation \eqref{eq:characteristic-1}, \eqref{eq:boundary-identity} together with Cauchy-Schwarz inequality, the boundary condition of \eqref{eq:boundary-value-problem} and the property of $\left. f_j\right|_{\partial M}$, we have
\begin{equation}\label{eq:3.18}
\begin{aligned}
\int_{M\setminus\Omega}\left(\frac{1}{2}f_j^2+1\right)|u_j|^2 {\rm vol}_{M} & \leq \int_{M\setminus \Omega}(2f_j^2-|\nabla f_j|)|u_j|^2 {\rm vol}_{M} \\
& \leq\int_{\Omega}(|\nabla f_j|-2f_j^2)|u_j|^2 {\rm vol}_{M}.
\end{aligned}
\end{equation}
Note that $\max_{\Omega}|u_j|\neq 0$, otherwise this estimate implies that $u_j$ vanishes globally. Therefore we may assume that $\max_{\Omega}|u_j|=1$ by appropriate rescaling. From \eqref{eq:characteristic-2} and \eqref{eq:3.18},
\begin{equation}\label{eq:bound-of-spinor}
\int_{\Omega}|u_j|^2 {\rm vol}_{M}+ \int_{M\setminus\Omega}\left(\frac{1}{2}f_j^2+1\right)|u_j|^2 {\rm vol}_{M} \leq \left(\frac{m\Lambda_c}{4c(m-1)\beta_2}+1\right)|\Omega|.
\end{equation}
Let 
\[
  \Upsilon_j=\min\limits_{\partial M}\left(\beta_2f_j-\frac{m}{2}|H|\right),  
\]
then $\Upsilon_j\to +\infty$ as $j\to +\infty$. Applying \eqref{eq:3.13}, \eqref{eq:characteristic-2} and (i) in Lemma \ref{lem:relative-GL-pair}, we have 
\begin{equation}\label{eq:bound-of-first-derivative}
\begin{aligned}
&\frac{m}{4c(m-1)}\int_{M}\big|\nabla|u_j|\big|^2 {\rm vol}_{M} +\int_{\partial M}\Upsilon_j|u_j|^2 {\rm vol}_{\partial M}  \\
\leq &\int_{M}\langle u_j,(-\beta_2f_j^2-\beta_2\nabla f_j\cdot \sigma)u_j\rangle {\rm vol}_{M}\\
&\qquad -\frac{m}{m-1}\int_{M}\left\langle u_j, \frac{{\rm scal}_M}{4}u_j+\mathscr{R}^{\mathcal{E}\oplus\mathcal{F}}u_j\right\rangle {\rm vol}_{M}
\leq   C_1
\end{aligned}
\end{equation}
for some constant $C_1$ independent of $j$. 

Arguing as in \cite[Section~3]{hirsch2023spectral}, from \eqref{eq:bound-of-spinor} and \eqref{eq:bound-of-first-derivative}, the sequence $|u_j|$ is uniformly bounded in $H^1(M)$. Thus $|u_j|$ weakly subconverges to a function $|u|$ in $H^1(M)$ with strong convergence in $H^s(M)$ for any $s\in [\frac{1}{2},1)$.
Since the trace map $\tau: H^s(M)\to H^{s-\frac{1}{2}}(\partial M)$ is continuous,
$|u_j|$ converges subsequentially to $\tau(|u|)$ in $L^2(\partial M)$. However, since $\Upsilon_j\to +\infty$ we find that \eqref{eq:bound-of-first-derivative} yields $\tau(|u|)=0$ on $\partial M$, and hence $|u|\in H_0^1(M^{\circ})$. 

Let $U=\left\{p\in M\big|{\rm d}_p\Phi\neq 0\right\}$ and $U_{\delta}=\left\{p\in U\big| \, \|\Lambda^2{\rm d}\Phi\|(p)<\delta\right\}$ for suitable $0<\delta<1$. Then taking the limit in \eqref{eq:3.13} while applying weak lower semi-continuity of the $H^1$-norm, the definition of the $c$-spectral constant, $\mathscr{R}^{\mathcal{E}\oplus\mathcal{F}}=0$ on $M\setminus U$ together with \eqref{eq:characteristic-2}, we have 
\begin{equation}\label{eq:estimate-1}
\begin{aligned}
0\geq &-\liminf_{j\to +\infty}\int_{\partial M}\left(\beta_2f_j+\frac{m}{2}H\right)\tau(|u_j|)^2 {\rm vol}_{\partial M}\\
\geq &\liminf_{j\to +\infty}\int_{M}\left(\frac{m}{4c(m-1)}(\big|\nabla|u_j|\big|^2+c\ {\rm scal}_M|u_j|^2)\right.\\
&\qquad +\left. \left(\beta_2f_j^2-\beta_2|\nabla f_j|\right)|u_j|^2+\langle u_j, \frac{m}{m-1}\mathscr{R}^{\mathcal{E}\oplus\mathcal{F}}u_j\rangle\right) {\rm vol}_{M}\\
\geq&\liminf_{j\to +\infty}\int_{M}\left[ \left(\beta_2f_j^2-\beta_2|\nabla f_j|+\frac{m\Lambda_c}{4c(m-1)}\right)|u_j|^2 \right. \\
& \qquad  +\left.\langle u_j, \frac{m}{m-1}\mathscr{R}^{\mathcal{E}\oplus\mathcal{F}}u_j\rangle \right] {\rm vol}_{M}\\
\geq &\liminf_{j\to +\infty}\int_{U}\frac{m\Lambda_c}{4c(m-1)}|u_j|^2+\langle u_j, \frac{m}{m-1}\mathscr{R}^{\mathcal{E}\oplus\mathcal{F}}u_j\rangle {\rm vol}_{M}\\
& \qquad +\liminf_{j\to +\infty}\int_{M\setminus U}C_{\varepsilon}|u_j|^2 {\rm vol}_{M}
\end{aligned}
\end{equation}
and continuing the estimate using Fatou's lemma and strong convergence in $L^2$ leads to
\begin{align*}
\geq & \int_{U}\left(\frac{m\Lambda_c}{4c(m-1)}|u|^2+\langle u, \frac{m}{m-1}\mathscr{R}^{\mathcal{E}\oplus\mathcal{F}}u\rangle\right){\rm vol}_{M} +\int_{M\setminus U} C_{\varepsilon}|u|^2 {\rm vol}_{M}\\
\geq &\int_{U} \frac{m}{4(m-1)}\underbrace{\left(\frac{\Lambda_c}{c}-({\rm scal}_N\circ \Phi)\cdot \|\Lambda^2{\rm d}\Phi\|(p)\right)}_{\geq 0\ \text{by}\ \text{(iv) and}\ \eqref{eq:c-spectral-constant} }|u|^2{\rm vol}_{M} +\int_{M\setminus U} C_{\varepsilon}|u|^2 {\rm vol}_{M}\\
\geq & \int_{\Omega\cap U_{\delta}} \frac{m}{4(m-1)} \left(\frac{\Lambda_c}{c}-C_{\delta}\right)|u|^2{\rm vol}_{M} +\int_{\Omega\setminus U} C_{\varepsilon} |u|^2{\rm vol}_{M}.
\end{align*}
Choosing $\delta$ a sufficiently small $\delta$ such that $C_{\delta}=({\rm scal}_N\circ \Phi)\cdot \|\Lambda^2{\rm d}\Phi\|(p)<\frac{\Lambda_c}{c}$ for all $p\in \Omega \cap U_{\delta}$. We conclude that the last two integrands are positive. Thus we arrive at a contradiction since $\max_{\Omega}|u|=1$. It follows that ${\rm deg}_{\widehat{A}}(\Phi)=0$. This completes the proof of Theorem \ref{thm:main}.
\end{proof}

\subsection{Proof of Theorem \ref{thm:spectral-geodesic-collar-neighborhood}}\label{subsec:Proof-Main-Theorem-2}
After all technical preparations, we are now ready to give a proof of Theorem \ref{thm:spectral-geodesic-collar-neighborhood}.
\begin{proof}[Proof of Theorem \ref{thm:spectral-geodesic-collar-neighborhood}]
For all sufficiently small $d>0$ denoted by $\mathcal{N}_d$ the open geodesic collar neighborhood of $\partial M$ of width $d$. Suppose, by contradiction, that $\mathcal{N}_{\ell}$ exists and there exists an $\varepsilon>0$ such that 
\begin{equation}
\ell \geq \omega:=\pi\sqrt{\frac{c}{\Lambda_c}\left(\frac{(4c-1)m+2-4c}{(4c-1)m+1}\right)}+\varepsilon.
\end{equation}
Fix $\Lambda\in (\omega, \ell)$. Then $K_{\Lambda}:=M\setminus \mathcal{N}_{\Lambda}$ is a compact manifold with boundary. For each $j$, we can define a Lipschitz function $f_j$ as in Equation \eqref{eq:defn-potential}, which satisfy a certain differential inequality and have the property that $f_j\to +\infty$ at $\partial M$ as $j\to +\infty$. Observe that $f_j=0$ on $K_{\Lambda}$. 
Note that the double ${\rm d}M$ has infinite $\widehat{A}$-area if and only if $(M,\partial M)$ has infinite relative $\widehat{A}$-area \cite[Remark~6.6]{cecchini2022scalar}. Then from \cite[Lemma~6.9]{cecchini2022scalar} and \eqref{eq:GGL}, there exists a Gromov-Lawson pair $(\mathcal{E},\mathcal{F})$ and an associated Dirac bundle $S\to M$ such that
\begin{itemize}
\item[(i)] $(\mathcal{E},\mathcal{F})$ and thus $S$ have support $K_{\Lambda}$;
\item[(ii)] $4\|\mathscr{R}^{\mathcal{E}\oplus \mathcal{F}}\|_{\infty}<\frac{\Lambda_c}{c}$ on $K_{\Lambda}^{\circ}$;
\item[(iii)] $\frac{m}{m-1}\|\mathscr{R}^{\mathcal{E}\oplus \mathcal{F}}\|_{\infty}<C_{\varepsilon}$ on $M\setminus K_{\Lambda}^{\circ}$;
\item[(iv)] ${\rm ind}(\mathcal{B}_{f_j,-1})={\rm indrel}(M;\mathcal{E},\mathcal{F})\neq 0$ for any potential $f_j$.
\end{itemize}
By (iv), we have a non-trivial element $u_j\in \ker(\mathcal{B}_{f_j,-1})$. Therefore, we argue as in the proof of Theorem \ref{thm:main}, and then from equation \eqref{eq:estimate-1} we have
\begin{align*}
0 \geq&\liminf_{j\to +\infty}\int_{M}\left(\beta_2f_j^2-\beta_2|\nabla f_j|+\frac{m\Lambda_c}{4c(m-1)}\right)|u_j|^2+\langle u_j, \frac{m}{m-1}\mathscr{R}^{\mathcal{E}\oplus\mathcal{F}}u_j\rangle {\rm vol}_{M}\\
\geq &\liminf_{j\to +\infty}\int_{K_{\Lambda}^{\circ}}\left( \frac{m\Lambda_c}{4c(m-1)}|u_j|^2+\langle u_j, \frac{m}{m-1}\mathscr{R}^{\mathcal{E}\oplus\mathcal{F}}u_j\rangle \right) {\rm vol}_{M}\\
&\qquad +\liminf_{j\to +\infty}\int_{M\setminus K_{\Lambda}^{\circ}}\left( C_{\varepsilon}|u_j|^2+\langle u_j, \frac{m}{m-1}\mathscr{R}^{\mathcal{E}\oplus\mathcal{F}}u_j\rangle\right) {\rm vol}_{M}
\end{align*}
and continuing the estimate using Fatou's lemma and strong convergence in $L^2$ leads to
\begin{align*}
\geq & \int_{K_{\Lambda}^{\circ} }\left(\frac{m\Lambda_c}{4c(m-1)}|u|^2+\langle u, \frac{m}{m-1}\mathscr{R}^{\mathcal{E}\oplus\mathcal{F}}u\rangle\right) {\rm vol}_{M}\\
&\qquad +\int_{M\setminus K_{\Lambda}^{\circ} } \left(C_{\varepsilon}|u|^2+ \langle u, \frac{m}{m-1}\mathscr{R}^{\mathcal{E}\oplus\mathcal{F}}u\rangle \right) {\rm vol}_{M}\\
\geq &\int_{K_{\Lambda}^{\circ} } \frac{m}{m-1}\underbrace{\left(\frac{\Lambda_c}{4c}-\|\mathscr{R}^{\mathcal{E}\oplus\mathcal{F}}\|_{\infty}\right)}_{>0\ \text{by (ii)} }|u|^2 {\rm vol}_{M} \\
& \qquad +\int_{M\setminus K_{\Lambda}^{\circ}} \underbrace{\left(C_{\varepsilon}-\frac{m}{m-1}\|\mathscr{R}^{\mathcal{E}\oplus\mathcal{F}}\|_{\infty}\right)}_{>0 \ \text{by (iii)} }|u|^2 {\rm vol}_{M}\\
\geq & \int_{\Omega\cap K_{\Lambda}^{\circ}} \left(\frac{\Lambda_c}{4c}-\|\mathscr{R}^{\mathcal{E}\oplus\mathcal{F}}\|_{\infty}\right) |u|^2 {\rm vol}_{M} \\
& \qquad +\int_{\Omega\setminus K_{\Lambda}^{\circ}} \left(C_{\varepsilon}-\frac{m}{m-1}\|\mathscr{R}^{\mathcal{E}\oplus\mathcal{F}}\|_{\infty}\right) |u|^2 {\rm vol}_{M}.
\end{align*}
Therefore we obtain a contradiction because $\max_{\Omega}|u|=1$. This finishes the proof of Theorem \ref{thm:spectral-geodesic-collar-neighborhood}.
\end{proof}

\section*{Acknowledgements}
The author would like to thank Prof. Zhenlei Zhang and Prof. Bo Liu for their careful reading, helpful comments, and many encouragements. The author also would like to thank the anonymous referee for valuable suggestions improving the exposition.

\def\cprime{$'$}
\providecommand{\bysame}{\leavevmode\hbox to3em{\hrulefill}\thinspace}
\providecommand{\MR}{\relax\ifhmode\unskip\space\fi MR }
\providecommand{\MRhref}[2]{%
  \href{http://www.ams.org/mathscinet-getitem?mr=#1}{#2}
}
\providecommand{\href}[2]{#2}

\end{document}